\documentclass{amsart}
\usepackage[utf8]{inputenc}
\usepackage{amssymb}
\usepackage{amsmath}
\usepackage{amsthm, amsfonts, mathrsfs}
\usepackage{amsmath}
\usepackage{amssymb}
\usepackage{amsthm}
\usepackage{graphicx}

\usepackage{amsrefs}
\newtheorem{thm}{Theorem}[section]
\newtheorem{cor}[thm]{Corollary}

\newtheorem{prop}[thm]{Proposition}

\newtheorem{qu}[thm]{Question}
\newtheorem{ex}[thm]{Example}
\newtheorem{df}[thm]{Definition}

\newcommand{\K}{\mathcal{K}}
\newcommand{\R}{\mathbb{R}}

\newcommand{\diag}{\Delta}

\newcommand{\te}{=_T}
\newcommand{\tq}{\ge_T}

\newcommand{\ctm}{\mathfrak{c}}

\begin{document}
\title{$P$-bases and Topological Groups}
\author[Z. Feng]{Ziqin Feng}
\address{Department of Mathematics and Statistics\\Auburn University\\Auburn, AL~36849}
\email{zzf0006@auburn.edu}
\date{August 2020}
\begin{abstract} A topological space $X$ is defined to have a neighborhood $P$-base at any $x\in X$ from some poset $P$ if there exists a neighborhood base $(U_p[x])_{p\in P}$ at $x$ such that $U_p[x]\subseteq U_{p'}[x]$ for all $p\geq p'$ in $P$.  We prove that a compact space is countable, hence metrizable, if it has countable scattered height and a $\mathcal{K}(M)$-base for some separable metric space $M$. This gives a positive answer to Problem 8.6.8 in \cite{Banakh2019}.

Let $A(X)$ be the free Abelian topological group on $X$. It is shown that if $Y$ is a retract of $X$ such that the free Abelian topological group $A(Y)$ has a $P$-base and $A(X/Y)$ has a $Q$-base, then $A(X)$ has a $P\times Q$-base. Also if $Y$ is a closed subspace of $X$ and $A(X)$ has a $P$-base, then $A(X/Y)$ has a $P$-base.

It is shown that any Fr\'{e}chet-Urysohn topological group with a $\mathcal{K}(M)$-base for some separable metric space $M$ is first-countable, hence metrizable. And if $P$ is a poset with calibre~$(\omega_1, \omega)$ and $G$ is a topological group with a $P$-base, then any precompact subset in G is metrizable, hence $G$ is strictly angelic. Applications in function spaces $C_p(X)$ and $C_k(X)$ are discussed.  We also give an example of a topological Boolean group of character $\leq \mathfrak{d}$ such that the precompact subsets are metrizable but $G$ doesn't have an $\omega^\omega$-base if  $\omega_1<\mathfrak{d}$.  This gives a consistent negative answer to Problem 6.5 in \cite{GKL15}. \end{abstract}

\date{\today}


\maketitle
\section{Introduction}
Let $P$ be a partially ordered set (poset). A topological space $X$ is defined to have a neighborhood $P$-base at any $x\in X$ if there exists a neighborhood base $(U_p[x])_{p\in P}$ at $x$ such that $U_p[x]\subseteq U_{p'}[x]$ for all $p\geq p'$ in $P$. We say that a topological space has a $P$-base if it has a neighborhood $P$-base at each $x\in X$. Thus, a topological group $G$ has a $P$-base if and only if it has a neighborhood $P$-base at the identity. Therefore, the classical metrization theorem of Birkhoff and Kakutani can be restated as `a Hausdorff topological group $G$ is metrizable if and only if $G$ has an $\omega$-base, i.e. is first-countable'. All the topological spaces and groups in this paper are assumed to be Tychonoff.

The concept of $\omega^\omega$-base with the original name `$\mathfrak{G}$-base'  was formally introduced in \cite{FKLS06} for studying (DF)-spaces, function spaces, and spaces in the
 class $\mathfrak{G}$ (see \cite{KKL11} for the definition). Here, $\omega^\omega$ is the family of all sequences of natural numbers in the pointwise partial order, i.e. $f\leq g$ if and only if $f(n)\leq g(n)$ for all $n\in\omega$.   Topological spaces and groups with an $\omega^\omega$-base were defined and studied in  \cite{GKL15}, for a systematic study see \cite{Banakh2019}. Works about topological algebra objects with an $\omega^\omega$-base can also be found in \cite{BL18}, \cite{LPT17}, and \cite{SF20}.

For any $x$ in the topological space $X$, the collection $\mathcal{T}_x(X)$ of open neighborhood of $x$ is a poset in the ordering defined by $U\leq V$ if and only if $U\supseteq V$. Throughout all the note, the local base at any point is considered as a poset defined in this way.  To understand this poset $\mathcal{T}_x(X)$, it is critical to investigate its confinality.  We use Tukey order to compare the cofinal complexity of posets. Tukey order \cite{Tuk40} was originally introduced, early in the 20th century, as a tool to
understand convergence in general topological spaces; however it was quickly seen to have broad
applicability in comparing posets.  Given two directed sets $P$ and $Q$, we say $Q$ is a Tukey quotient of $P$, denoted by $P\geq_T Q$, if there is a map $\phi:P\rightarrow Q$ carrying cofinal subsets of $P$ to cofinal subsets of $Q$.  It is known that a topological space $X$ has a $P$-base if and only if $\mathcal{T}_x(X)\leq_T P$ for each $x\in X$.

This paper is organized as follows. In Section~\ref{scs}, we prove that a compact space is countable, hence metrizable, if it has countable scattered height and a $\mathcal{K}(M)$-base for some separable metric space $M$. Hence this is also true for a compact space with countable scattered height and an $\omega^\omega$-base, which gives a positive answer to Problem 8.6.8 in \cite{Banakh2019}. It is worth mentioning that in \cite{DF20} the authors prove that under the assumption $\omega_1<\mathfrak{b}$, any compact space with an $\omega^\omega$-base is metrizable  and any scattered compact space with an $\omega^\omega$-base is countable. We also prove that $D^2\setminus \Delta$  is strongly dominated (see definition in Section~\ref{scs}) by the Bowtie space $B$, here $D$ is the double arrow space. This gives negative answer to Questions 4.4-4.13 in \cite{GTk}.

 For a Tychonoff space $X$, let $A(X)$ be the free Abelian topological group over $X$, in the sense that there is a continuous mapping $\sigma$ from $X$ to $A(X)$ such that: i) $\sigma(X)$ topologically generates the topological group $A(X)$; and ii), for every continuous mapping $f$ from $X$ to an Abelian group $H$, there is a continuous homomorphism $\Tilde{f}$ from $A(X)$ to $H$ such that $\Tilde{f}\circ \sigma=f$.  It is proved in \cite{Gar20} and independently in \cite{BL18} that $A(X)$ has an $\omega^\omega$-base if and only if the universal uniformity $\mathcal{U}_X$ of $X$ has an $\omega^\omega$-ordered base. In Section~\ref{fag} we study the $P$-bases of free Abelian topological groups through quotient spaces and quotient groups. We prove that if $Y$ is a retract of $X$ such that $A(Y)$ has a $P$-base and $A(X/Y)$ has a $Q$-base, then $A(X)$ has a $P\times Q$-base. We also prove that if $Y$ is a closed subspace of $X$ and $A(X)$ has a $P$-base, then $A(X/Y)$ has a $P$-base.

 In Section~\ref{tg_sa_m} we show that any Fr\'{e}chet-Urysohn topological group with a $\mathcal{K}(M)$-base for some separable metric space $M$ is first-countable, hence metrizable. Also, we prove that if $P$ is a poset with calibre~$(\omega_1, \omega)$ and $G$ is a topological group with a $P$-base, then any precompact subset in G is metrizable, hence, $G$ is strictly angelic. In this section we apply our results to function spaces $C_p(X)$ and $C_k(X)$. We also give an example of a topological Boolean group of character $\leq \mathfrak{d}$ such that the precompact subsets are metrizable and $G$ doesn't have an $\omega^\omega$-base if $\omega_1<\mathfrak{d}$. This gives a consistent negative answer to Problem 6.5 in \cite{GKL15}.

  \section{Preliminaries}

 Given that $P$ and $Q$ are both Dedekind complete (every bounded subset has a least upper bound),  $P\geq_T Q$ if and only if there is a map $\phi: P\rightarrow Q$ which is order-preserving and such that $\phi(P)$ is cofinal in $Q$.  Throughout the note, all the posets are Dedekind complete unless otherwise stated. If $P$ and $Q$ are mutually Tukey quotients, we say that $P$ and $Q$ are Tukey equivalent, denoted by $P=_T Q$.

For any separable metric space $M$, $\mathcal{K}(M)$ is the collection of compact subsets of $M$ ordered by set-inclusion. Fremlin observed that if a separable metric space $M$ is locally compact, then $\mathcal{K}(M)=_T\omega$. Its unique successor under Tukey order is the class of Polish but not locally compact spaces. For $M$ in this class, $\mathcal{K}(M)=_T\omega^\omega$ where $\omega^\omega$ is ordered by $f\leq g$ if $f(n)\leq g(n)$ for each $n\in \omega$. In \cite{GM16}, Gartside and Mamataleshvili constructed a $2^{\mathfrak{c}}$-sized antichain in $\mathcal{K}(\mathcal{M})=\{\mathcal{K}(M): M\in \mathcal{M}\}$ where $\mathcal{M}$ is the class of separable metric spaces. 

Let $P$ be a directed poset, i.e., for any points $p, p'\in P$, there exists a point $q\in P$ such that $p\leq q$ and $p'\leq q$.  A subset $C$ of $P$ is \emph{cofinal} in $P$ if for any $p\in P$, there exists a $q\in C$ such that $p\leq q$. Then $\text{cof}(P)=\min\{|C|: C\text{ is cofinal in }P\}$. We also define $\text{add}(P)=\min\{|Q|: Q \text{ is unbounded in }P\}$. For any $f, g\in \omega^\omega$, we say that $f\leq^\ast g$ if the set $\{n\in\omega:f(n)>g(n) \}$ is finite. Then $\mathfrak{b}=\text{add}(\omega^\omega, \leq^\ast)$ and $\mathfrak{d}=\text{cof}(\omega^\omega, \leq^\ast)$. See \cite{Douwen84} for more information about small cardinals.

Let $\kappa\geq \mu\geq \lambda$ be cardinals. We say that a poset $P$ has \emph{calibre~$(\kappa, \mu, \lambda)$} if for every $\kappa$-sized subset $S$ of $P$ there is a $\mu$-sized subset $S_0$ such that every $\lambda$-sized subset of $S_0$ has an upper bound in $P$. We write calibre~$(\kappa, \mu, \mu)$ as calibre~$(\kappa, \mu)$ and calibre~$(\kappa, \kappa, \kappa)$ as calibre~$\kappa$. It is known that $\mathcal{K}(M)$ has calibre~$(\omega_1, \omega)$ for any separable metric space $M$, hence so does $\omega^\omega$. And if $P$ has calibre~$(\omega_1, \omega) $ and $Q$ is a Turkey quotient of $P$, then $Q$ also has calibre~$(\omega_1, \omega)$. Under the assumption $\omega_1=\mathfrak{b}$, $\omega_1$ is a Tukey quotient of $\omega^\omega$ (see \cite{AMETD} for the discussion of the spectrum of $\omega^\omega$), hence $\omega_1+1$ has an $\omega^\omega$-base. However, when $\omega_1<\mathfrak{b}$, $\omega_1$ is not a Tukey quotient of $\omega^\omega$, hence the space $\omega_1+1$ doesn't have an $\omega^\omega$-base, but it has a $\mathcal{K}(\mathbb{Q})$-base (see \cite{DF20}), where $\mathbb{Q}$ is the space of rational numbers.

We start with some basic properties of the property calibre $(\omega_1, \omega)$ and $P$-bases. For any two posets $(P, \leq_P)$ and $(Q, \leq_Q)$, the ordering on $P\times Q$ is naturally defined as follows: $(p, q)\leq (p', q')$ if and only if $p\leq_P p'$ and $q\leq_Q q'$. For $i=1, 2$, let $\pi_i$ be the projection mapping from $P\times Q$ to $P$ or $Q$ respectively.

\begin{prop}\label{cw1w} If $P$ is a poset with calibre~$(\omega_1, \omega)$ and $Q$ is a countable poset, then $P\times Q$ also has calibre~$(\omega_1, \omega)$.  \end{prop}

\begin{proof} Let $S$ be an $\omega_1$-sized subset of $P\times Q$. Then there is a $q\in Q$ and an $\omega_1$-sized subset $S'$ of $S$ such that $\pi_2(x)=q$ for all $x\in S'$. Since $P$ has calibre~$(\omega_1, \omega)$, $\{\pi_1(x): x\in S'\}$ is an $\omega_1$-sized subset of $P$, hence it has a countable subset bounded in $P$. Therefore, there is a countable subset of $S$ which is bounded in $P\times Q$. \end{proof}

\begin{prop}\label{op_cont} Suppose that $X$ has a $P$-base for some poset $P$. If there is an open continuous mapping $f$ from $X$ onto $Y$, then $Y$ has a $P$-base.
\end{prop}

\begin{proof} Fix $y\in Y$. Choose $x\in f^{-1}(y)$. Then there is a local $P$-base $\{U_p[x]: p\in P\}$ at $x$. We claim that $\{f(U_p[x]):p\in P\}$ is a local $P$-base at $y$. Clearly $f(U_p[x])$ is an open neighborhood of $y$ since $f$ is open.  For any open neighborhood $V$ of $y$, $f^{-1}(V)$ is an open neighborhood of $x$. Therefore, there is a $p\in P$ such that $U_p[x]\subseteq f^{-1}(V)$, i.e., $f(U_p[x])\subseteq V$.  \end{proof}

\section{(Scattered) Compact Spaces}\label{scs}

We recall that a topological space $X$ is scattered if each non-empty subspace of $X$ has an isolated point. The complexity of a scattered space can be determined by the scattered height.

For any subspace $A$ of a space $X$, let $A'$ be the set of all non-isolated
points of $A$. It is straightforward to see that $A'$ is a closed subset of $A$. Let $X^{(0)} = X$ and define $X^{(\alpha)} = \bigcap_{\beta<\alpha} (X^{(\beta)} ) '$ for each $\alpha > 0$. Then a space $X$ is  scattered if $X^{(\alpha)} = \emptyset$ for some ordinal $\alpha$. If $X$ is scattered, there exists a unique ordinal $h(x)$ such that $x\in X^{(h(x))}\setminus X^{(h(x)+1)}$ for each $x\in X$. The ordinal $h(X)=\sup\{h(x): x\in X\}$  is called the scattered
height of $X$ and is denoted by $h(X)$. It is known that any compact scattered space is zero-dimensional. Also, it is straightforward to show that for any compact scattered space $X$, $X^{(h(X))}$ is a non-empty finite subset.

Let $X$ be a compact space. If $X$ has an $\omega^\omega$-base and finite scattered height, then $X$ is proved in \cite{Banakh2019} to be countable, hence metrizable. It is proved in \cite{DF20} that the same result also holds if $X$ has finite scattered height and a $P$-base for some poset $P$ with calibre~$(\omega_1, \omega)$. The compact space $\omega_1+1$ has an $\omega^\omega$-base under the assumption $\omega_1<\mathfrak{b}$. The authors in \cite{DF20} proved that if $\omega_1<\mathfrak{b}$ any compact scattered space is countable, hence metrizable. Hence it is natural to ask, as in \cite{Banakh2019}, whether  a compact space $X$ is metrizable in ZFC if it has an $\omega^\omega$-base and {\em countable} scattered height. Next, we give a positive answer to this question by proving a more general result.

\begin{prop}  Any compact space is countable, hence metrizable, if it has a countable scattered height and a $\mathcal{K}(M)$-base, where $M$ is a separable metric space.  \end{prop}

\begin{proof} Let $X$ be a compact space with countable scattered height $\alpha$ and a $\mathcal{K}(M)$-base for some separable metric space $X$.  We'll prove by induction on $\alpha$.

Assume any compact space with scattered height $\beta$ is countable for any $\beta<\alpha$. Without loss of generality, we assume that $X^{(\alpha)}$ is a singleton, denoted by $x^\ast$. For any open neighborhood $U$ of $x^\ast$, $X\setminus U$ is a scattered compact space with scattered height $<\alpha$, hence is countable. So it is sufficient to prove that $X$ is first-countable at $x^\ast$.

In \cite{DF20}, the authors prove that any compact space with countable tightness and a $\mathcal{K}(M)$-base is first-countable where $M$ is a separable metric space. Hence it suffices to prove that $X$ has countable tightness. For each $x\in X\setminus \{x^\ast\}$ there is a clopen neighborhood $U_x $ of $x$ such that $U_x\cap X^{(h(x))}=\{x\}$ since any scattered compact space is zero-dimensional. Clearly $U_x$ is a compact space with scattered height $<\alpha$. By the inductive assumption $U_x$ is countable. So $X$ has countable tightness at any $x$ in $X\setminus \{x^\ast\}$.

Next, we prove that $X$ has countable tightness at $x^\ast$. Let $B$ a subset of $X\setminus \{x^\ast\}$ with $x^\ast\in \overline{B}$.  Let $\beta_\ast=\sup\{\gamma: X^{(\gamma)}\cap (\overline{B}\setminus \{x^\ast\})\neq\emptyset \text{ and }\gamma<\alpha\}$. We will show there is a countable subset $C\subseteq B$ such that $x^{\ast}\in \overline{C}$ in two cases: 1) $\alpha$ is a limit ordinal; and 2) $\alpha$ is a successor.

Firstly, we assume that $\alpha$ is a limit ordinal. If $\beta_\ast<\alpha$, then $\overline{B}$ is a compact space with scattered height  $\leq \beta_\ast+1$, hence it is countable. Then $C=B$ works. Suppose $\beta_\ast=\alpha$. Then there exists a confinal sequence $\gamma_n$ in $\alpha$ such that $X^{(\gamma_n)}\cap (\overline{B}\setminus \{x^\ast\})\neq\emptyset$. For each $n$, pick $x_n\in  X^{(\gamma_n)}\cap (\overline{B}\setminus \{x^\ast\})$. Since $X$ is compact, $x^\ast$ is in the closure of $\{x_n: n\in \omega\}$. For each $n\in\omega$, let $C_n$ be a countable subset of $B$ such that $x_n\in \overline{C_n}$. Therefore $C=\bigcup\{C_n: n\in \omega\}$ works.

From now on, we assume that $\alpha$ is a successor and let $\alpha_-$ be the ordinal such that $\alpha=\alpha_-+1$. Note that $X^{(\alpha_-)}$ is a compact space with scattered height $1$, hence it is countable. If $\beta_\ast<\alpha_-$, then the closure of $B$ is a compact space with scattered height $<\alpha$. Hence $\overline{B}$ is countable, so is $B$. Then let  $C=B$ which is clearly countable.

Suppose $\beta_\ast=\alpha_-$. If $X^{(\alpha_-)}\cap \overline{B}$ is finite, choose a clopen neighborhoood $U$ of $x^\ast$ such that  $U\cap ( X^{(\alpha_-)}\cap \overline{B})= \{x^\ast\}$. Then $U\cap \overline{B}$ is a compact subspace with scattered height $<\alpha$, hence countable. Therefore, $C=U\cap  B$ works. If $X^{(\alpha_-)}\cap \overline{B}$ is infinite, list it as $\{x_n: n\in \omega\}$; clearly $x^\ast\in \overline{\{x_n: n\in \omega\}}$ since $X$ is compact. By the inductive assumption, we can pick a countable subset $C_n$ of $B$ such that $x_n$ is in the closure of $C_n$. Then $C=\bigcup C_n$ is a countable subset of $B$ with $x^\ast$ in the closure of $C$. This finishes the proof.       \end{proof}

The corollary below answers Problem 8.6.8 in \cite{Banakh2019} affirmatively.

\begin{cor} A compact space is countable (hence metrizable) if and only if  it has countable scattered height and an $\omega^\omega$-base. \end{cor}

We don't know the answer to the following question. The answer is proved to be positive in \cite{DF20} if the space has a finite scattered height.

\begin{qu} Let $P$ be a poset with cablire~$(\omega_1, \omega)$ and  $X$ a compact space with a $P$-base. Is $X$ countable if $X$ has a countable scattered height? \end{qu}

We say a space $Z$ is dominated by $Y$ if there exists a family $\mathcal{D}=\{D_K: D_K \text{ is a compact subset of }Z \text{ and } K\text{ is a compact subset of }Y\}$ such that $Z=\bigcup \mathcal{D}$ and $D_K\subset D_L$ if $K\subset L$. Furthermore, $Z$ is said to be strongly dominated by $Y$ if the family $\mathcal{D}$ is also confial in $\mathcal{K}(Z)$. It is shown  that any compact space $X$ with $X^2\setminus \Delta$ being dominated by $\mathbb{Q}$ is metrizable in \cite{Feng2019} and the same result holds for a compact space $X$ with $X^2\setminus \Delta$ being dominated by $M$ for some separable metric space $M$, see \cite{San2020}, here $\Delta=\{(x,x): x\in X\}$. Questions 4.4-4.13 in \cite{GTk} ask whether a compact space $X$ is metrizable if $X^2\setminus \Delta$ is dominated by other topological space. The following example answers these questions negatively.

\begin{ex}\label{ex_dom} \textnormal{\textbf{(Gartside)}}
Let $B$ denote the Bowtie space.

Then $\K (B) \te [\ctm]^{<\omega}$. Hence for every compact space, $X$, of weight $\le \ctm$, we have $\K(B) \tq \K(X^2 \setminus \diag)$.

In particular, there is a $\sigma$-compact cosmic space, $B$, which strongly dominates  $D^2  \setminus \diag$, where $D$ is the Double Arrow space -- a compact, first countable but non-metrizable space.
\end{ex}

\begin{proof}
Let $B$ be the standard Bowtie space obtained by refining the Euclidean topology  on $\R^2$ by declaring sets of the form  $B(x, s)=\{(x, 0)\}\cup \{(x', y): |y|<s|x-x'|\}$ to be open for all $x, s\in \mathbb{R}$ with $s > 0$.



Let $K_0=\{(a,b) : a \in [-1,1], \ b =a^2\}$. Observe that $K_0$ is a compact subset of $B$. For each $x$ in $\R$, let $K_x$ be $K_0$ translated to $(x,0)$, i.e., $K_x=\{(x, y): (x, y)=(x,0)+(a, b) \text{ for some }(a, b)\in K_0\}$.

Then $\mathcal{C} =\{K_x : x \in \R\}$ has size $\ctm$ but no infinite subfamily has an upper bound in $\K(B)$, i.e., has union whose closure is compact.
To see this, suppose for a contradiction that infinite family of $K_x$'s is contained in a compact set, say $K_\infty$. Note that the set of corresponding $x$'s in $\R$ is infinite and has compact closure.
Passing to a subset we may suppose that we have an infinite family of $K_{x_n}$'s where the $x_n$'s are (without loss of generality) an increasing sequence in $\R$ converging to some $x_\infty$.
But now, for each $n$, $K_{x_n}$ meets the line $\{x_\infty\} \times \R$ in exactly one point, say $(x_\infty,y_n)$.
And now we see that $\{ (x_\infty,y_n) : n \in \mathbb{N}\}$ is an infinite closed discrete subset of compact $K_\infty$ -- contradiction.

Hence $\K(B)$ does not have calibre $(\ctm, \omega)$. Then $\K(B) \tq [\ctm]^{<\omega}$ by \cite{AMETD}*{Lemma 11}. As $\K(B)$ has size $\ctm$, clearly $\K(B)\te [\ctm]^{<\omega}$.

Let $X$ be a compact space with  $w(X) \le \ctm$.  Then $\K(X^2 \setminus \diag)$ has cofinality no more than $\ctm$ (because it is order isomorphic to the neighborhoods of the diagonal, and a cofinal family of these comes from finite covers by basic open sets). Then by \cite{AMETD}*{Lemma 6}, $[\ctm]^{<\omega}\geq_T \K(X^2 \setminus \diag)$, hence $\K(B)\geq_T \K(X^2 \setminus \diag)$. Therefore $X^2 \setminus \diag$ is strongly dominated by the Bowtie space $B$.
\end{proof}

\section{Free Abelian Topological Groups}\label{fag}
A property $\mathcal{P}$ is said to be a {\em three-space property} if for every topological group $G$ and  a closed normal subgroup $H$ of $G$ both $H$ and $G/H$ having $\mathcal{P}$ implies that $G$ has $\mathcal{P}$. It is a classical result that having an $\omega$-base is a three-space property (see \cite{ArhTka2008}). The authors ask in \cite{GKL15} whether having an $\omega^\omega$-base is a three-space property. We will start with a sufficient condition on $(G, H)$ which guarantees that $G$ has an $\omega^\omega$-base if both $H$ and $G/H$ have an $\omega^\omega$-base.

\begin{df}Let $G$ be a topological group and $H$ a closed normal subgroup of $G$. We say that $(G, H)$ is a good pair if there is a mapping $p$ from a local base $\mathcal{B}$ at the identity in  $H$ to the topology on $G$ such that:
\begin{itemize}
    \item[i)] $p(B)\cap H=B$ for any $B\in\mathcal{B}$;
    \item[ii)] for any $B, B'\in \mathcal{B}$, $p(B)^2\subseteq p(B')$ if $B^2\subseteq B'$.
    \item[iii)] for any $B, B'\in \mathcal{B}$, $p(B)^{-1}\subseteq p(B')$ if $B^{-1}\subseteq B'$.
\end{itemize}

\end{df}

Our next result shows that if $H$ is a metrizable closed normal subgroup of $G$, then $(G, H)$ is a good pair. The proof is straightforward.

\begin{prop}\label{mgp}Let $G$ be a topological group with the identity $e$. If $H$ is a closed normal subgroup of $G$ such that $\mathcal{T}_H(e)\leq_T\omega$, then $(G, H)$ is a good pair. \end{prop}

\begin{thm}\label{tg_rbase} Let $G$ be a topological group and $H$ a closed normal subgroup of $G$ such that $(G, H)$ is a good pair. Suppose that $P$, $Q$, and $R$ are Dedekind-complete posets such that $P\times Q\leq_T R$.  If $H$ has a $P$-base and $G/H$ has a $Q$-base, then $G$ has an $R$-base.
\end{thm}

\begin{proof} Let $\phi$ be an order-preserving mapping from $R$ to $P\times Q$ such that $\phi(R)$ is cofinal in $P\times Q$. Let $e$ be the identity element in $G$.

Since $(G, H)$ is a good pair and $H$ has a $P$-base, there is a collection $\{W_p: p\in P\}$ of open neighborhoods of $e$ in $G$ such that: i) $\{W_p\cap H: p\in P\}$ is a $P$-base of $H$; ii) if $(W_p\cap H)^2\subset W_{p'}$ then $W_p^2\subset W_{p'}$; and iii) if $(W_p\cap H)^{-1}
\subset W_p'\cap H$ then $W_p^{-1}\subset W_{p'}$. Let $\{V_q: q\in Q\}$ be a symmetric $Q$-base of $G/H$. Let $q$ be the quotient mapping from $G$ to $G/H$. For each $r\in R$, define $U_r=W_{\pi_1(\phi(r))}\cap q^{-1}(V_{\pi_2(\phi(r))})$. By the order-preserving property of $\phi$, it is straightforward to verify that $U_r\supset U_{r'}$ given $r\leq r'$.

We show that $\mathcal{U}=\{U_r: r\in R\}$ is a base for $G$. Take an open neighborhood $U$ of the identity element $e$. Then, we pick an open symmetric neighborhood $V$ of the unit such that $VV\subset U$. Fix $p_1\in P$ such that $W_{p_1}\cap H\subset V$. By conditions ii) and iii) above,  there exists a $p_2\in P$ such that  $p_2\geq p_1$ and  $W_{p_2}^{-1}W_{p_2}\subset W_{p_1}$. Pick $q_1\in Q$ such that $q(V_{q_1}) \subset q(V\cap W_{p_2})$. Since $\phi(R)$ is cofinal in $P\times Q$, there exists an $r\in R$ such that $\pi_1(\phi(r)) \geq p_2$ and $\pi_{2}(\phi(r))\geq q_1$.

We claim that $U_r\subset U$. Pick $g\in U_r$. By the definition of $U_r$ and the choice of $r$, $g\in W_{p_2}$ and $g\in q^{-1}(V_{\pi_2(\phi(r))})\subset V_{q_1}\cdot H\subseteq  V\cap W_{p_2} \cdot H$. Hence $g=ab$ for some $a\in V\cap W_{p_2} $ and $b\in H$. Therefore $b=a^{-1}g\in W^{-1}_{p_2}W_{p_2}\subset W_{p_1}$. Finally, we obtain that $g\in (V\cap W_{p_2})\cdot (W_{p_1}\cap H)\subset V\cdot V\subset U$.

Therefore, $\mathcal{U}$ is an $R$-base for the topological group $G$. \end{proof}

Therefore, if $H$ is a metrizable closed normal subgroup of $G$, $(G, H)$ is a good pair by Proposition~\ref{mgp}. Let $M$ be any noncompact separable metric space.  Note that $\mathcal{K}(M)\geq_T \mathcal{K}(M)\times \omega$. Then if $G/H$ has a $\mathcal{K}(M)$-base, then $G$ has a $\mathcal{K}(M)$-base. The case $M=\omega^\omega$ is proved in \cite{GKL15}*{Proposition 2.9}.

Next we discuss good pairs of the form $(A(X), A(Y, X))$, where $A(Y, X)$ is the subgroup of $A(X)$ generated by $Y$. It is shown in \cite{ArhTka2008}*{Theorem 7.4.5} that $A(Y, X)$ is a closed subgroup of $A(X)$ if $Y$ is a closed subspace of $X$. Also, if $Y$ is a retract of $X$, then $A(Y, X)$ is topologically isomorphic to $A(Y)$, in which case we consider them to be same.

\begin{prop}\label{suf_gp} Let $G$ be a topological group and $H$ a closed normal subgroup of $G$. If there is a continuous homomorphism from $G$ to  $H$ whose restriction on $H$ is identity mapping, $(G, H)$ is a good pair. \end{prop}

\begin{proof} Let $\phi$ be a continuous homomorphism from $G$ to $H$ such that $\phi(h)=h$ given $h\in H$. Fix a local base $\mathcal{B}$ of the identity in $H$. For each $B\in \mathcal{B}$, define $p(B)=\phi^{-1}(B)$. It is clear that $p(B)\cap H=B$ for each $B\in \mathcal{B}$ since $\phi(h)=h$ for any $h\in H$.

Pick $B, B'\in \mathcal{B}$ such that $B^2\subseteq B'$. Choose $g$ and $g'$ in $p(B)$. We show that $g\cdot g'\in p(B)$. By the definition of $p(B)$, there exists $h, h'\in B$ such that $\phi(g)=h$ and $\phi(g')=h'$. Hence $\phi(g\cdot g')=\phi(g)\cdot \phi(g')=h\cdot h'\in B'$. Therefore, $g\cdot g'\in \phi^{-1}(B')=p(B')$.

Pick $B, B'\in \mathcal{B}$ such that $B^{-1}\subset B'$. Choose $g\in p(B)$. Then there exists $h\in B$ such that $\phi(g)=h$, then $\phi(g^{-1})=h^{-1}$. Since  $B^{-1}\subset B'$, $\phi(g^{-1})=h^{-1}\in B'$, i.e., $g^{-1}\in p(B')$.  This finishes the proof. \end{proof}

\begin{prop} Let $X$ be a Tychonoff space and $Y$ a retract of $X$. Then $(A(X), A(Y,X))$ is a good pair.  \end{prop}

\begin{proof} Since $Y$ is a retract of $X$, $Y$ is a closed subspace of $X$. By Theorem 7.4.5 in \cite{ArhTka2008}, $A(Y, X)$ is a closed subgroup of $A(X)$.

Let $f$ be a continuous mapping from $X$ to $Y$ such that $f(y)=y$ for each $y\in Y$. By Corollary 7.1.9 in \cite{ArhTka2008}, $f$ admits an extension to a continuous homomorphism $A(f)$ from $A(X)$ to $A(Y)$. Clearly for any $h\in A(Y)$, $A(f)(h)=h$.  The identity mapping from $A(Y)$ to $A(Y, X)$ is a topological isomorphism. Hence, we obtain a continuous homomorphism from $A(X)$ to $A(Y, X)$ whose restriction on $A(Y, X)$ is the identity mapping. By Proposition~\ref{suf_gp}, $(A(X),A(Y, X))$ is a good pair.
\end{proof}

Let $Y$ be a closed subspace of $X$. Define an equivalence relation $\sim$ on $X$ by $x\sim y$ if and only if $x, y\in Y$. We denote the quotient space $X/\sim$ by $X/Y$; there is a canonical quotient mapping from $X$ to $X/Y$. We denote the equivalence class $[y]$ for any $y\in Y$ by $[Y]$ which clearly is a retract of $X/Y$. So in the following discussion, we use $A([Y])$ to represent the subgroup $A([Y], X/Y)$.  Next, we investigate the relations of $A(X)$, $A(Y)$, and $A(X/Y)$ in terms of $P$-bases.

\begin{prop}\label{quot} Let $Y$ be a closed subspace of a Tychonoff space $X$. Then $A(X/Y)/A([Y])$ is topologically isomorpic to $A(X)/A(Y, X)$.
\end{prop}

\begin{proof} Let $f$ be the quotient map from $X$ to $X/Y$. By Corollary 7.1.9 in \cite{ArhTka2008}, $f$ admits an extension to an open continuous homomorphism $A(f)$ from $A(X)$ to $A(X/Y)$. The canonical mapping $\pi$ from $A(X/Y)$ to $A(X/Y)/A([Y])$ is an open continuous mapping. Hence we obtain an open continuous mapping $\phi=\pi\circ A(f)$ from $A(X)$ to the quotient group $A(X/Y)/A([Y])$.

Let $e$ be the identity in $A(X/Y)/A([Y])$. Clearly, $\phi^{-1}(e)=A(Y, X)$. Then by Theorem 1.5.13 in \cite{ArhTka2008}, $A(X)/A(Y, X)$ is topologically isomorpic to $A(X/Y)/A([Y])$. \end{proof}

\begin{thm}\label{q_p}Let $Y$ be a closed subspace of a Tychonoff space $X$. If the free Abelian group $A(X)$ has a $P$-base, then $A(X/Y)$ also has a $P$-base. \end{thm}

\begin{proof} Since $Y$ is closed, the subgroup $A(Y, X)$ is closed and normal. The canonical mapping from $A(X)$ to $A(X)/A(Y, X)$ is open, continuous, and onto. Then by Proposition~\ref{op_cont}, $A(X)/A(Y, X)$ has a $P$-base. By Proposition~\ref{quot}, $A(X/Y)/A([Y])$ has a $P$-base. Note that $A([Y])$ is a discrete subspace of $A(X/Y)$. So $(A(X/Y), A([Y]))$ is a good pair by Proposition~\ref{mgp}. Applying Theorem~\ref{tg_rbase}, $A(X/Y)$ has a $1\times P$-base, i.e., a $P$-base. \end{proof}

The following example is a direct application of Theorem~\ref{q_p}.

\begin{ex} Let $S_n=\{x_{m, n}: m\in \omega\}\cup \{x_n\}$ be a nontrivial convergent sequence with $x_n$ being the limit point for each $n$ and let $T$ be the topological sum of the family $\{S_n: n\in \omega\}$. Note that $T$ is a metrizable space whose subspace of non-isolated points is $\sigma$-compact. Hence by Theorem 1.2 in \cite{BL18}, the free Abelian group $A(T)$ has an $\omega^\omega$-base.

Let $S_\omega$ be the quotient space of $T$ obtained by identifying all the limit points to a singleton. Then by Theorem~\ref{q_p}, $A(S_\omega)$ has an $\omega^\omega$-base. \end{ex}

We don't know whether the same result holds for free topological groups.

\begin{qu} Let $Y$ be a closed subspace of a Tychonoff space $X$. Suppose that the free group  $F(X)$ has a $P$-base. Does  $F(X/Y)$ have a $P$-base?
\end{qu}

Next we show that if $Y$ is a retract of $X$, then the bases of $A(Y)$ and $A(X/Y)$ at the identity are both Turkey quotients of some base of $A(X)$ at identity.
\begin{thm} Suppose that $P$, $Q$, and $R$ are Dedekind-complete posets such that $P\times Q\leq_T R$. Let $X$ be a Tychonoff space and $Y$  a retract of $X$.

If  $A(Y)$ has a $P$-base and $A(X/Y)$ have a $Q$-base, then $A(X)$ has an $R$-base.  \end{thm}
\begin{proof} Because $Y$ is a retract of $X$, $(A(X), A(Y, X))$ is a good pair by Proposition~\ref{suf_gp}. So it is sufficient to show that $A(X)/A(Y, X)$ has a $Q$-base. Then $A(X)$ has an $R$-base by Theorem~\ref{tg_rbase}.

Next we show that $A(X)/A(Y, X)$ has a $Q$-base. Since $A(X/Y)$ have a $Q$-base and the canonical mapping $\pi$ from $A(X/Y)$ to $A(X/Y)/A([Y])$ is an open continuous mapping, $A(X/Y)/A([Y])$ has a $Q$-base by Proposition~\ref{op_cont}. By Proposition~\ref{quot}, $A(X)/A(Y, X)$ is topologically isomorphic to $A(X/Y)/A([Y])$. Hence $A(X)/A(Y, X)$ has a $Q$-base. This finishes the proof. \end{proof}

For every directed poset $P$, it is straightforward to verify that $P\times P=_T P$. Hence we obtain the following corollary.

\begin{cor} Let $P$ be  any Dedekind-complete directed poset.  Suppose that $X$ is a Tychonoff space and $Y$ is a retract of $X$. If both $A(Y)$ and $A(X/Y)$ have a $P$-base, then $A(X)$ has a $P$-base. \end{cor}

The result above holds clearly for all the posets in the form $\mathcal{K}(M)$ for some topological space $M$. Therefore, the free Abelian topological group $A(X)$ has an $\omega^\omega$-base if $X$ is a Tychonoff space and $Y$ is a retract of $X$ such that  both $A(Y)$ and $A(X/Y)$  have an $\omega^\omega$-base.

\section{Topological Groups}\label{tg_sa_m}

\begin{df} A space $X$ is said to satisfy Property (AS) if the following holds:
\begin{description}
\item[(AS)] for any family $\{x_{m, n}: m, n\in \omega\}\subseteq X$ with $\lim_m x_{m, n}=x$ for all $n\in \omega$, there exist two strictly increasing sequences of natural numbers $\{m_k: k\in \omega\}$ and $\{n_k: k\in \omega\}$ such that $\lim_k x_{m_k, n_k}=x$.
\end{description}
\end{df}

It is proved that in \cite{KKL11} any Fr\'{e}chet-Urysohn topological group satisfies Property (AS).
\begin{thm} Let $P=\mathcal{K}(M)$ for a separable metric space $M$. If $G$ is a Fr\'{e}chet-Urysohn topological group with a $P$-base, then $G$ is first-countable, hence metrizable. \end{thm}
\begin{proof} Let $e$ be the identity element in $G$ and  $\mathcal{U}=\{U_p: p\in P\}$ be an open $P$-base at the identity $e$.

For any separable metric space $M$, $P=\mathcal{K}(M)$ endowed with the Vietoris topology is also separable metrizable, hence second countable. Also if $p_n$ converges to $p$ in $P$, then $p^\ast=p\cup (\bigcup\{p_n\})$ is also compact, hence it is an element in $P$ with $p_n\subset p^\ast$ and $p\subset p^\ast$.

Fix a countable base $\{B_n: n\in\omega\}$ of $P$. For each $p\in P$  and $n\in\omega$ with $p\in B_n$, define $D_n(p)=\bigcap \{U_q: q\in B_n\}$. For each $p$, we list an increasing sequence of natural numbers $\{n^p_i: i\in \omega\}$ such that $p\in B_{n^p_i}$ for each $i$ and $\{B_{n^p_i}: i\in \omega\}$ is a decreasing base of $P$.

First, we claim that for each $p\in P$ there exists $i$ such that $D_{n^p_{i}}(p)$ is a neighborhood of the identity $e$. Suppose for contradiction that there exists $p\in P$ such that $D_{n^p_{i}}(p)$ is not a neighborhood of $e$ for each $i\in \omega$. Hence $e$ is in the closure of $G\setminus D_{n^p_{i}}(p)$ for each $i\in \omega$. Since $G$ is Fr\'{e}chet-Urysohn, there exists a sequence $\{x_{m, i}\}$ in  $G\setminus D_{n^p_{i}}$ for each $i$ such that $\lim_m x_{m, i}=e$. Then we apply Property (AS) to get strictly increasing sequences of natural numbers $\{m_k: k\in\omega\}$ and $\{i_k:k\in\omega\}$ such that $\lim_k x_{m_k, i_k}=e$.  For each $k\in \omega$, pick $q_k\in B_{n^p_{i_k}}$ such that $x_{m_k, i_k}\notin U_{q_k}$. Clearly $q_k$ converges to $p$ in $P$. Hence $p^\ast=p\cup \bigcup\{q_k: k\in \omega\}$ is also compact in $M$, hence it is in $\mathcal{K}(M)$. Since $q_k\leq p^\ast$ for each $k$, $U_{q_k}\supset U_{p^\ast}$. Therefore  $x_{m_k, i_k}\notin U_{p^\ast}$ for each $k\in\omega$. This contradicts with  $\lim_k x_{m_k, i_k}=e$  because $U_{p^\ast}$ is an open neighborhood of $e$.

For each $p\in U$, fix the minimal natural number $i_p$ such that $D_{n^p_{i_p}}(p)$ is a neighborhood of $e$. The family $\{D_{n^p_{i_p}}(p): p\in P\}$ is clearly countable. By the definition of $D_n(p)$, the family $\{\text{int}(D_{n^p_{i_p}}(p)): p\in P\}$ is a countable base at the identity $e$. Hence,  $G$ is first-countable.    \end{proof}

\begin{qu} Suppose $G$ is a Fr\'{e}chet-Urysohn topological group such that a local base at the identity has Calibre~$(\omega_1, \omega)$. Is $G$ first-countable?  \end{qu}

A topological space is angelic if every relatively countably compact set $A$ in the space has a compact closure, i.e., relatively compact, and for each $x$ in the closure of $A$ there exists a sequence in $A$ converging to $x$. Here, a subset $A$ of a space $X$ is relatively countably compact if every sequence in $A$ has a cluster-point in $X$.  A space is said to be strictly angelic if it is angelic and each separable compact subset is first-countable. It is known that any topological group is a uniform space and in uniform spaces, it is shown in \cite{Fl80} that any relatively countably compact subset is precompact. A subset $B$ of a topological group $G$ is said to be precompact in $G$ if for any neighbourhood $U$ of the identity in $G$ there is a finite subset set $F\subset G$ such that $B\subset FU$ and $B\subset UF$. It is known \cite{ArhTka2008} that any precompact set is relatively compact in Ra\v{i}kov complete topological groups.


Let $K$ be a compact space. Gartide and Morgan show in \cite{GMor16} that $K$ is compact if the off-diagonal subspace $K^2\setminus \Delta$ has a compact cover $\mathcal{D}$ with Calibre~$(\omega_1, \omega)$ which swallows all the compact sets of  $K^2\setminus \Delta$, i.e., $\mathcal{D}$ is cofinal in $\mathcal{K}(K^2\setminus \Delta)$.

The proof of the following result is very similar to Theorem 3.9 in \cite{GKL15}.

\begin{thm}\label{sa} Let $P$ be any directed set with Calibre~$(\omega_1, \omega)$.  If $G$ is a topological group with a $P$-base, then every precompact subset $K$ in $G$ is metrizable. Consequently, $G$ is strictly angelic. \end{thm}

\begin{proof} Similar to Proposition 2.7 in \cite{GKL15}, the Ra\v{i}kov completion of any group with a $P$-base also has a $P$-base. So we can assume $G$ is Ra\v{i}kov complete.  Hence it is sufficient to show that every compact subset in $G$ is metrizable. Let $K$ be a compact subset of $G$ and $\{U_p: p\in P\}$ be a $P$-base at the identity $e$ in $G$.

For each $p\in P$, define $L_p=\{(x, y)\in K\times K: xy^{-1}\notin U_p\}$. Note that each $L_p$ is closed, hence compact in $K\times K$. We show that $\mathcal{L}=\{L_p: p\in P\}$ is a $P$-ordered compact cover of  the off-diagonal space $K\times K\setminus \Delta$ which swallows all its compact sets. It is straightforward to verify that $L_p\subset L_{p'}$ given $p\leq p'$ in $P$.  Let $C$ be a compact subset of $K\times K\setminus \Delta$. Then $\hat{C}=\{xy^{-1}: (x, y)\in C\}$ is a compact subset of $G$ which doesn't contain the identity $e$. Since $\{U_p: p\in P\}$ is a local base at the identity $e$, there exists a $p_0\in P$ such that $U_{p_0}\subset G\setminus \hat{C} $, furthermore, $C\subset L_{p_0}$. Therefore,  the collection $\mathcal{L}$ swallows all the compact subsets of $K\times K\setminus \Delta$. Then by \cite{GMor16}, $K$ is metrizable. This finishes the proof.  \end{proof}

So next, we discuss some applications of Theorem~\ref{sa} in function spaces. For any Tychonoff space $X$, let $C(X)$ be the space of continuous real-valued functions. We use $C_p(X)$ to  denote $C(X)$ equipped with the pointwise convergence topology in which a basic neighborhood of $f\in C(X)$ is $B(f, F, \epsilon)=\{g: |f(x)-g(x)|<\epsilon \text{ for any } x \in F\}$ where $F$ is a finite subset of $X$ and $\epsilon$ is a positive real number.

The following result is proved in \cite{GMor19}. In fact, it is shown in \cite{GMor19} that the neighborhood filter at the identity element $\textbf{0}$ in $C_p(X)$ is Tukey equivalent to $[X]^{<\omega}$.

\begin{prop} Let $P$ be a poset with Caliber~$(\omega_1,\omega)$ with $P\geq_T \omega$. Then the following are equivalent:
\begin{itemize}
\item[i)] $C_p(X)$ has a $P$-base;
\item[ii)] $C_p(X)$ is first-countable, i.e., has an $\omega$-base;
\item[iii)] $X$ is countable.
\end{itemize}

\end{prop}



For any Tychonoff space $X$, it is known (see \cite{KKL11}) that $C_p(X)$ is angelic if  $X$ is web-compact. A topological space $X$ is said to be web-compact if there exists a nonempty subset $\Sigma$ of $\omega^\omega$ and a family $\{A_{\alpha}\}$ of subsets of $X$ such that the following two conditions hold:
\begin{itemize}
    \item[i)] $X=\overline{\bigcup \{A_{\alpha}: \alpha\in \Sigma\}}$;
    \item[ii)] the sequence $(x_k)_k$ has a cluster point in $X$ if $x_k\in C_{\alpha|_k}$ for each $k$, here $C_{\alpha|_k}=\bigcup \{A_\beta: \beta(i)=\alpha(i) \text{ for all }i< k\}$.
\end{itemize}

We use $C_k(X)$ to  denote $C(X)$ equipped with the compact-open topology in which a basic neighborhood of $f\in C(X)$ is $B(f, K, \epsilon)=\{g: |f(x)-g(x)|<\epsilon \text{ for any } x \in K\}$ where $K$ is a compact subset of $X$ and $\epsilon$ is a positive real number.

Let $M$ be a metric space. It is known  (see Corollary 6.10 in \cite{KKL11}) that $C_k(M)$ is angelic if and only if $M$ is separable. Here, we get a stronger result below.
\begin{cor} Let $M$ be a metric space. Then the following are equivalent:
\begin{itemize}
    \item[i)] $M$ is separable;
    \item[ii)] $C_k(M)$ is strictly angelic;
    \item[iii)] $C_k(M)$ is angelic.
\end{itemize}
\end{cor}
\begin{proof} The implication of ii)$\rightarrow$iii) is clear and  iii)$\rightarrow$i) is proved in \cite{KKL11}. It remains to prove i)$\rightarrow$ii).

Let $M$ be a separable metric space. Then $\mathcal{B}=\{B(\textbf{0}, K, 1/n): K\in \mathcal{K}(M) \text{ and }n \in\mathbb{N}\}$ is a local base at the identity $\textbf{0}$ in $C_k(M)$. Clearly, $\mathcal{B}=_T \mathcal{K}(M)\times \omega$ which is known to have Calibre~$(\omega_1, \omega)$. Hence by Theorem~\ref{sa}, $C_k(M)$ is strictly angelic.\end{proof}



It is proved in \cite{KKL11} that $C_k(\omega_1)$ has an $\omega^\omega$-base under the assumption  $\omega_1=\mathfrak{b}$, but not when $\omega_1<\mathfrak{b}$. Note that $\mathcal{K}(\omega_1)=_T\omega_1$, hence it has Calibre~$(\omega_1, \omega)$. Note that $C_{k}(\omega_1)$ has a $\mathcal{K}(\omega_1)\times \omega$-base, hence an $\omega_1\times \omega$-base. Note that $\omega_1$ has Calibre~$(\omega_1, \omega)$, hence  so does $\omega_1\times \omega$. Therefore, $C_k(\omega_1)$ is strictly angelic by Theorem~\ref{sa}. Also, using the fact $\omega_1\leq_T\omega^\omega\times \omega_1\leq_T\mathcal{K}(\mathbb{Q})$ proved in \cite{GM17}, it has a $\mathcal{K}(\mathbb{Q})$-base which clearly has Calibre~$(\omega_1, \omega)$.

\begin{ex} The function space $C_k(\omega_1)$ has a $\mathcal{K}(\mathbb{Q})$-base, hence it is strictly angelic. \end{ex}

Motivated by Theorem~\ref{sa}, it is natural to ask as in  \cite{GKL15}*{Problem 6.5}  whether it is sufficient for a topological group $G$ to have an $\omega^\omega$-base if all precompact subsets in $G$ are metrizable and the character of $G$ is $\leq \mathfrak{d}$. Next we give an example which gives a consistent negative answer to this problem.

\begin{ex}\label{c_e} There is a topological Boolean group of character $\chi(G)\leq \mathfrak{d}$ such that all the precompact subsets in $G$ are metrizable and $G$ doesn't have an $\omega^\omega$-base consistently, but it has a $\mathcal{K}(\mathbb{Q})$-base in ZFC.
\end{ex}

\begin{proof} Let $X=2^{\omega_1}$ with the $G_\delta$-topology, i.e., a basic open neighborhood of any $x\in X$ is $B_\alpha(x)=\{y: y(\gamma)=x(\gamma) \text{ for all }\gamma<\alpha\}$ for $\alpha<\omega_1$. Define the operation on $X$ as $(x+y)(\alpha)=x(\alpha)+y(\alpha)$ for each $\alpha<\omega_1$ following the rules: $0+0=0$, $0+1=1+0=1$, and $1+1=0$. It is straightforward to verify that $X$ is a topological Boolean group.  Also, any countable subset of  $X$ is closed and discrete, hence any precompact subset in $X$ is finite, therefore metrizable.

Let $\mathbf{0}$ be the identity element in $X$, i.e., $\mathbf{0}(\alpha)=0$ for each $\alpha<\omega_1$. For each $\alpha<\omega_1$, define $B_\alpha(\textbf{0})=\{y: y(\gamma)=0 \text{ for all }\gamma<\alpha\}$. Then $\{B_\alpha(\textbf{0}): \alpha<\omega_1\}$ is a local base at $\textbf{0}$. Hence the topological group $X$ has an $\omega_1$-base and its character is $\leq \mathfrak{d}$. Under the assumption $\mathfrak{b}>\omega_1$, $\omega_1$ is not a Tukey quotient of $\omega^\omega$, hence $X$ doesn't have an $\omega^\omega$-base. Using the fact that $\mathcal{K}(\mathbb{Q})\geq_T\omega_1$ again, $X$ has a $\mathcal{K}(\mathbb{Q})$-base in ZFC. \end{proof}


The answer to the following question is still unknown.
\begin{qu}Let $G$ be a topological group of character $\chi(G)\leq \mathfrak{d}$. If all precompact subsets in $G$ are metrizable, does $G$ admit an $\omega^\omega$-base under the assumption $\omega_1=\mathfrak{b}$? \end{qu}

\begin{ex} There is a $\sigma$-compact Fr\'{e}chet-Urysohn topological Boolean group which is not metrizable. \end{ex}

\begin{proof} Let $X=\{x: x\in 2^{\omega_1} \text{ and } \text{supp}(x) \text{ is finite}\}$, here supp$(x)=\{\alpha: x(\alpha)\neq 0\}$. The operation on $X$ is same as the one defined in Example~\ref{c_e}. Clearly, $X$ is a Boolean group. Consider $X$ as a subspace of $2^{\omega_1}$ equipped with the product topology. Then the space $X$ is Fr\'{e}chet-Urysohn, but not metrizable.

Next we show that $X$ is $\sigma$-compact. Define $K_n=\{x: x\in X\text{ and } \text{supp}(x)\leq n\}$. It suffices to show that $K_n$ is a closed subset of $2^{\omega_1}$ for each $n\in \omega$. Fix $n\in\omega$ and $x\notin K_n$.  Choose $\{\alpha_1, \ldots, \alpha_{n+1}\}$ such that $x(\alpha_i)=1$ for each $i=1, \ldots, n+1$. Let $B(x)=\{y: y(\alpha_i)=x(\alpha_i)\text{ for each }i=1, \ldots, n+1\}$ which is an open neighborhood of $x$. And, $K_n\cap B(x)=\emptyset$. Hence $K_n$ is compact for each $n$. Therefore, $X=\bigcup\{K_n: n\in \omega\}$ is $\sigma$-compact. \end{proof}


\textbf{Acknowledgements} I am very grateful to Professors Gary Gruehage and Paul Gartside for their valuable comments and suggestions which improve the paper. I would also like to thank Professor Paul Gartside for Example~\ref{ex_dom} and the anonymous reviewer for his/her valuable suggestions and corrections.

\end{document}